\documentclass[12pt,a4paper]{amsart}

\usepackage[utf8]{inputenc}
\usepackage{multicol}
\usepackage{amsmath}
\usepackage{amssymb}
\usepackage{amsfonts}
\usepackage{amsthm}
\usepackage[all]{xy}
\usepackage{tikz}
\usepackage{hyperref}
\usepackage[shortlabels]{enumitem}

\theoremstyle{plain}
\newtheorem*{atw*}{Theorem}
\newtheorem{atw}{Theorem}[section]
\newtheorem*{alemat*}{Lemma}
\newtheorem{alemat}[atw]{Lemma}
\newtheorem{pro}[atw]{Proposition}
\newtheorem*{pro*}{Proposition}
\newtheorem{cor}[atw]{Corollary}
\newtheorem*{cor*}{Corollary}

\theoremstyle{definition}
\newtheorem{adf}[atw]{Definition}
\newtheorem*{adf*}{Definition}

\theoremstyle{remark}
\newtheorem{rem}{Remark}
\newtheorem*{rem*}{Remark}
\newtheorem{ex}{Example}
\newtheorem*{ex*}{Example}

\newcommand{\Z}{\mathbb{Z}}
\newcommand{\N}{\mathbb{N}}
\newcommand{\Q}{\mathbb{Q}}
\newcommand{\C}{\mathbb{C}}
\newcommand{\T}{\mathbb{T}}
\newcommand{\PP}{\mathbb{P}}

\newcommand{\ee}{{\bf e}}
\newcommand{\vv}{{\bf v}}

\newcommand{\xto}[1]{{\xrightarrow{#1}}}
\newcommand{\lin}{\operatorname{span}}
\newcommand{\rank}{\operatorname{rank}}

\title{Motivic Chern classes of configuration spaces}
\author{Jakub Koncki}
\address{Institute of Mathematics, University of Warsaw\\
Banacha 2, 02--097 Warsaw, Poland}
\email{j.koncki@mimuw.edu.pl}
\begin{document}
		
	\begin{abstract}
		We calculate the equivariant motivic Chern class for configuration space of a quasiprojective (maybe singular) variety and the space of vectors with different directions. We prove the formulas for generating series of these classes. We generalize the localization theorems results about Białynicki-Birula decomposition to acquire some stability for the motivic Chern classes of configuration spaces. 
	\end{abstract}

\maketitle	
	\section{Introduction}
	We consider complex quasiprojective varieties.
	Invariants of maps satisfying certain additivity property play an important role in the problem of generalization of characteristic classes to singular varieties.
	We are interested in such invariants $cl_*$ which assign to every map of quasiprojective varieties $f:X \to B$ an element of some generalized homology theory of $B$ and satisfy the following property:
	\begin{description}
		\item[The additivity property] For a closed subvariety $Y \subset X$ and its complement $U$, there is an equality 
		$$cl_*(f: X\to B)=cl_*(f_{|Y}: Y\to B)+cl_*(f_{|U}: U\to B)\,.$$
	\end{description}
	In the presence of a group action one can consider equivariant version of such classes with values in some generalized equivariant cohomology theory e.g. equivariant Chow groups, or equivariant K-theory of coherent sheaves. The action of an algebraic torus $\T=(\C^*)^r$ is of particular interest, due to the localization theorems.
	
	For a family of morphisms $f_i:X_i \to B_i$ indexed by natural numbers it is a common practice to consider generating series
	$$\sum_{i=0}^\infty cl_*(f_i)t^i$$
	in some appropriately defined ring of power series $R[[t]]$ (with $R$ a commutative ring with
	unit). It often turns out that the expression for whole generating series has simpler form than the individual components.
	
	One of the invariants satisfying additivity property is the motivic Chern class $mC$ and its homological counterpart the Hirzebruch
	class $T_{y*}$, which are introduced in \cite{BSY} (see also \cite{SYp} for a survey).
	It generalizes previously defined characteristic classes of singular varieties (e.g. the Chern-Schwartz-MacPherson class \cite{CSM}, \cite{Oh2} for the equivariant version).
	Lately the equivariant counterpart of the Hirzebruch class was
	defined in
	\cite{WeHir}
	as well as of the motivic Chern class in \cite{FRW,AMSS}.
	Consider quasiprojective complex varieties 
	with an algebraic action of a torus $\T=(\C^*)^r$.
	Denote the equivariant motivic Chern class by $mC^\T$. It assigns to every $\T$-equivariant map $f : X \to B$ of $\T$-varieties an element of $G^\T(B)[y]$, i.e. a polynomial in $y$ over the equivariant $K$-theory of coherent sheaves on $B$. It is uniquely defined by three properties (after \cite{FRW}, section 2.3):
\begin{description}
	\item[1. Additivity] Consider a closed $\T$-equivariant subvariety $Y \subset X$ and its complement~$U$. For any $\T$-equivariant map $f :X\to B$, there is an equality 
	$$mC^\T(f: X\to B)=mC^\T(f_{|Y}: Y\to B)+mC^\T(f_{|U}: U\to B)\,.$$
	
	\item[2. Functoriality] For a proper map $f:B\to B'$ we have $$mC^\T(X\stackrel{f\circ g}\to B')=f_*mC^\T(X\stackrel{g}\to B)\,.$$
	
	\item[3. Normalization] For a smooth variety $M$ we have $$mC^\T(id_M)=\lambda_y(T^*M)\,,$$
	where $\lambda_y$ is the Grothendieck $\lambda$ operation defined for any $\T$-vector bundle by formula
	$$\lambda_y(E):=\sum_{i=0}^{\rank E}[\Lambda^iE]y^i.$$
	
\end{description}
	For a smooth $\T$-variety $M$ the equivariant K-theory of coherent sheaves is isomorphic to the equivariant K-theory of vector bundles, which we denote by $K^\T(M)$. It allows to consider motivic Chern classes $mC^\T(X\to M)$ as elements of $K^\T(M)[y]$.
	
	The goal of this note is to compute and study behaviour of the (torus equivariant) motivic Chern classes of configuration spaces.
	We consider ordered configuration spaces 
		$$Conf_k(B)=\{(x_1,...,x_k)\in B^k| x_i\neq x_j\}$$
	of a quasiprojective (maybe singular) variety $B\subset M$ embedded
	(equivariantly) as a locally closed subset into an ambient
	smooth variety $M$,
	and the space of pairwise linearly independent nonzero vectors
	$$C_k(\C^n):= \{ (\vv_1,\dots , \vv_k)\in (\C^n)^k|\forall_i \vv_i\neq 0 \text{, } \lin (\vv_i) \neq \lin(\vv_j)\}\,.$$
	Both of these spaces are open varieties with singular complement. We consider only ordered configurations. For classes of symmetric products and unordered
	configuration spaces see \cite{CMOSSY,CMSSY,Bal,MSch,Oh} and others.
	We compute classes of inclusions
	$$
	Conf_k(B) \subset B^k\subset M^k \text{ and }	
	C_k(\C^n)\subset (\C^n)^k.
	$$
	In the subsequent sections we establish generating series and search for stability results. In section \ref{s:szer} we find a compact elegant form of several exponential
	generating series given by suitable exponents, e.g. for
	torus-equivariant motivic Chern classes localized at an isolated $\T$-fixed point $pt\in M$ in the ambient smooth $\T$-variety:
	\begin{atw*}
	The exponential generating series of the localized $\T$-equivariant motivic Chern classes
	$$\frac{mC^{\T}\left(Conf_k(B) \subset M^k\right)_{|pt^k}}{eu^{\T}\left(\{pt^k\} \subset M^k\right)}\in S^{-1}K^\T(pt^k)[y]$$
	where $eu^{\T}\left(\{pt^k\} \subset M^k\right)=\lambda_{-1}(T^*M^k)_{|pt^k}$ denotes the K-theoretical Euler class
	of the $\T$-equivariant normal bundle of the fixed point inclusion $\{pt^k\} \subset M^k$, is given by
		\begin{align*}
		&1+\sum_{k=1}^{\infty}\frac{t^k}{k!}\cdot
		\frac{mC^{\T}\left(Conf_k(B) \subset M^k\right)_{|pt^k}}{eu^{\T}\left(\{pt^k\} \subset M^k\right)}=
		\exp\left(\frac{mC^\T(B\subset M)_{|pt}}{eu^{\T}\left(\{pt\} \subset M\right)}\cdot\log(1+t)\right)
		\end{align*}
	\end{atw*}
	
	\begin{rem}
		Similar formulae also hold with the same proof for $\T$-equivariant MacPherson Chern classes $c_*^\T$
		and Hirzebruch classes $T_{y*}^\T$
		in (completed) equivariant (even degree) (co)homology or Chow groups,
		with the equivariant Euler class given by the corresponding top-dimensional
		equivariant Chern class of the normal bundle.
	\end{rem}
	For the orbit configuration space $C_k(\C^n)$ of pairwise linearly
	independent nonzero vectors in $\C^n$ we get similarly:
\begin{atw*}
	The exponential generating series of the classes $$\frac{mC^{\T_\alpha}\big(C_k(\C^n) \subset (\C^n)^k\big)}{eu^{\T_\alpha}\big(\{0\} \subset (\C^n)^k\big)}$$ is given by
	\begin{multline*}
	1+\sum_{k=1}^{\infty}\frac{t^k}{k!}\cdot
	\frac{mC^{\T_\alpha}\big(C_k(\C^n) \subset (\C^n)^k\big)}{eu^{\T_\alpha}\big(\{0\} \subset (\C^n)^k\big)} =
	\prod_{i=1}^{n}\exp\left( 
	\frac{\lambda_y(T^*\PP(\C^n))_{|\ee_i}}{\lambda_{-1}(T^*\PP(\C^n))_{|\ee_i}}
	\log\left(1+\frac{t(1+y)}{\alpha_i-1}\right)\right)
	\end{multline*}
	Here $\T_\alpha \subset GL_n(\C)$ is the corresponding diagonal subtorus,
	acting with weights $\alpha_1,...,\alpha_n$ on $\C^n$ and $\ee_1,...,\ee_n \in \PP(\C^n)$
	are the fixed points for the induced $\T$-action on $\PP(\C^n)$.
\end{atw*}
	In the section \ref{s:BB} we show that there is a connection between class of configuration space of a smooth, projective variety and the configuration space of its fixed points component, critical in Białynicki-Birula decomposition \cite{B-B1,B-B2,B-B3} for some one dimensional subtorus. Namely
	\begin{pro*}
		Let $\C^*$ be one dimensional subtorus of torus $\T$. Denote by $\T_1$ the quotient torus $\T/\C^*$. Let $M$ be a smooth, projective $\T$-variety. Let $F_1$ be the sink in the Białynicki-Birula decomposition according to the torus $\C^*$. Then
		$$
		mC^{\T_1}\left(Conf_k(F_1)\to F_1^k\right)=\lim_{{\bf t} \to 0}
		\left(mC^{\T}\left(Conf_k(M)\to M^k\right)_{|F_1^k}\right)
		$$
		for appropriately defined limit map (definition \ref{df:lim}). Analogues result holds when $F_1$ is the source.
	\end{pro*}
	To prove this statement we generalize the Lefschetz-Riemann-Roch theorem and the results of \cite{WeBB} connecting Białynicki-Birula decomposition with localization theorems to the relative case.
	\paragraph{\bf Tools and notations}
	We consider only complex quasiprojective varieties.
	\begin{itemize}
		\item  $\alpha_1,...,\alpha_n$ denote the weights of the diagonal torus in GL$_n(\C)$.
		\item $K^\T(X)$ denotes the algebraic $\T$-equivariant $K$-theory of vector bundles.
		\item The localized $K$-theory $S^{-1}K^\T(X)$ denotes the equivariant $K$-theory of $X$ localized in the multiplicative system $S:=K^\T(pt)\setminus\{0\}$.
		\item $\lambda_y$ denotes multiplicative characteristic class in $K$-theory defined by:
		$$ 	\lambda_y(E):=\sum_{i=0}^{\rank E}[\Lambda^iE]y^i. $$
		\item $eu(E)$ denotes the Euler class of the vector bundle $E$ in $K$-theory equal to $\lambda_{-1}(E^*)$.
		\item $K^\T(Var/X)$ denotes the Grothendieck group of $\T$-equivariant morphisms from quasiprojective varieties to the given $\T$-variety $X$  (cf. \cite{Loo,Bi,AMSS}).
		\item When we think of a morphism as an element of the Grothendieck group we put it in square brackets.
		\item We use an abbreviation BB-decomposition for the Białynicki-Birula decomposition.
	\end{itemize}

	In the presence of a torus action one can apply localization to the fixed points and Lefschetz-Riemann-Roch theorem to calculate global invariants.
	
	\begin{atw}[Lefschetz-Riemann-Roch, \cite{ChGi} theorem 5.11.7] \label{tw:LRR}
		Assume that a torus $\T$ acts on smooth varieties $X$ and $Y$. Consider the multiplicative system $S$ of nonzero elements in $K^\T(pt)$. Let $F \subset Y^{\T}$ be a component of fixed points. For any proper $\T$-equivariant map $f:X \to Y$ and element $\alpha \in S^{-1}K^\T(X)$ the pushforward $f_*\alpha$ can be computed using an equality
		$$
		\frac{i_F^*f_*\alpha}{eu^{\T}(\nu_F)}=\sum_{G\subset X^{\T}\cap f^{-1}(F)}f^{|G}_*\frac{i_G^*\alpha}{eu^{\T}(\nu_G)}\,.
		$$
		Where the sum is indexed by the fixed points components of $X$ which lie in preimage of $F$.
	\end{atw}
	
	\begin{rem} \label{rem:LRR}
		The Lefschetz-Riemann-Roch theorem is a consequence of the localization formula \cite{AtBo1,BeVe}. Namely
		$$
		i_F^*f_*\alpha=i_F^*f_*\sum_{G \subset X^{\T}}i^G_*\frac{i^*_G\alpha}{eu^\T(\nu_G)}
		=\sum_{G \subset X^{\T}}i_F^*i^{f(G)}_*f^{|G}_*\frac{i^*_G\alpha}{eu^\T(\nu_G)}
		=eu(\nu_F)\sum_{G\subset X^{\T}\cap f^{-1}(F)} \frac{i^*_G\alpha}{eu^\T(\nu_G)}\,,
		$$
		where morphisms $i$ are inclusions of fixed points set components. The first equality follows from the localization formula.
	\end{rem}

	\subsection{Acknowledgments}
	I would like to thank to Andrzej Weber for his support, guidance and helpful conversations. I thank the anonymous referee for pointing a way to generalize main theorem and for his comments which helped to improve overall quality of the paper.
	The author is supported by  the research project of the Polish National Research Center 2016/23/G/ST1/04282 (Beethoven 2, German-Polish joint project).
	
	\section{Computation of the motivic Chern classes}
	We consider a quasiprojective (maybe singular)
	variety $B$ with an action of an algebraic torus~$\T$
	embedded (equivariantly) as a locally closed subset into an ambient	smooth variety $M$.
	We consider its configuration space
	$$Conf_k(B)=\{(x_1,...,x_k)\in B^k| x_i\neq x_j\},$$
	with the natural diagonal action of the torus $\T$. The configuration space $Conf_k(B)$ is an invariant subset of smooth variety $M^k$. We consider also the space
	$$C_k(\C^n):= \{ (\vv_1,\dots , \vv_k)\in (\C^n)^k|\forall_i \vv_i\neq 0 \text{, } \lin (\vv_i) \neq \lin(\vv_j)\}$$
	of nonzero, pairwise linearly independent vectors (a bundle over the configuration space of projective space). This space admits the natural action of the group GL$_n(\C)\times (\C^*)^k$. Denote by $\T_\alpha$ the diagonal torus in the group GL$_n(\C)$, and by $\T_\beta$ the torus $(\C^*)^k$.
	Denote by $\alpha_1,...,\alpha_n$ and $\beta_1,...,\beta_k$ weights of the tori $\T_\alpha$ and $\T_\beta$.
	Our goal in this section is computation of the classes
	$
	mC^{\T}(Conf_k(B) \subset B^k)$
	and	
	 $mC^{\T_\alpha\times\T_\beta}(C_k(\C^n)\subset (\C^n)^k).$
	For simplicity we sometimes omit an ambient space in notation of the motivic Chern class.
	\begin{rem}
		\label{rem:conf}
		Spaces $C_k(\C^n)$ are orbit configuration spaces $F_{\C^*}(\C^n-\{0\},k)$ (defined in \cite{Xi} notation from 		\cite{FeZorb}) 
		 where the group $\C^*$ acts by homotheties. 
		It might look more natural to consider the bigger space~$F_{\C^*}(\C^n,k)$. Its class can be easily computed from the class of $F_{\C^*}(\C^n-\{0\},k)$ using the equality:
		$$F_{\C^*}(\C^n,k)=F_{\C^*}(\C^n-\{0\},k)\sqcup \bigsqcup_{i=1}^k F_{\C^*}(\C^n-\{0\},k-1)\,. $$
	\end{rem}
\subsection{Configuration space}

	Let $[k]$ denote the set $\{1,...,k\}$. Define a partition of a set $A$ as a set of nonempty, pairwise disjoint subsets, whose sum is the whole set $A$. Let $X_k$ denote the set of partitions of $[k]$. For a given partition $P\in X_k$ and $i \in [k]$ denote by $P(i)$ element of $P$ which contain $i$. 
	For a partition $P\in X_k$ consider the set 
	$$B_P=\{(x_1,...,x_k)\in B^k| x_i = x_j \text{ when $P(i)=P(j)$} \}\,. $$
	The configuration space $Conf_k(B)$ is given by $\binom{k}{2}$ conditions $x_i\neq x_j$. The inclusion-exclusion formula implies formula (for combinatorial details see appendix \ref{s:combi})
	\begin{align} \label{wyr:r1}
	[Conf_k(B)\to  M^k]=\sum_{P\in X_k}a(P)[B_P\to  M^k] \,,
	\end{align}
	where numbers $a(P) \in \Z$ depends only on partition $P$.
	Consider the set $G(P)$  of graphs  whose set of vertices is equal to $[k]$ and whose connected components induce partition $P$ on the set $[k]$ . For a given graph $G$ let $E(G)$ be its set of edges. Then (cf. lemma \ref{lem:aP})
	\begin{align}
	\label{wyr:r2}
	a(P)=\sum_{G\in G(P)} (-1)^{|E(G)|}=\prod_{P_i \in P} (-1)^{|P_i|-1}(|P_i|-1)!\,.
	\end{align}
	\begin{rem}
		Note that the equality (\ref{wyr:r2}) follows also from consideration of
		the M\"obius function of the partition lattice (see \cite{Stan}, example
		3.10.4).
	\end{rem}
	Thus to compute the classes $mC^\T(Conf_k(B)\to M^k)$ it is enough to compute the classes \hbox{$mC^\T(B_P\to M^k)$.} The set $B_P$ is  isomorphic to $B^{|P|}$ (such isomorphism is induced by a choice of ordering on the partition $P$). We have commutative diagram
	 $$\xymatrix{
	B^{|P|} \ar[r]^{i_P} \ar[d] & B^k \ar[d]\\
		M^{|P|} \ar[r]^{i_P}&M^k
 	}$$
	  where $i_P$ are "diagonal" maps corresponding to partition
$P$.
	  It follows from the functorial properties of the motivic Chern classes that
	 \begin{align}
	 \label{f1}
	 mC^{\T}(Conf_k(B)\subset M^k)=\sum_{P\in X_k} a(P)i_{P*}
	 \left(mC(B^{|P|}\subset M^{|P|})\right)
	 \,.
	 \end{align}
	 \begin{ex}
	 	Suppose that $B=M$. Then we have $mC^\T(B\to M)=\lambda_y(T^*M)$. Suppose moreover that $M=V$ is a complex vector space with a
	 	linear $\T$-action. Then under the Thom
	 	isomorphism  $K^\T(V) \simeq K^\T(\{0\})$ (cf. \cite{ChGi} theorem 5.4.17), given by inclusion $\{0\}\subset V$, we have $$mC^\T(id_V ) = \lambda_y(T^*V )=\lambda_y(V^*).$$
	 	Similarly
	 	$$mC(i_P:V^{|P|}\to V^k)=\lambda_y(V^*)^{|P|}\lambda_{-1}(V^*)^{k-|P|}$$
	 	since $i_P^*\circ i_{P*}$ is given by multiplication with the Euler class
	 	of the normal bundle of the diagonal embedding $i_P$
	 	(see \cite{ChGi} Prop. 5.4.10).
	Consider the affine space $\C^n$ with action of the diagonal torus. Denote by $\alpha_1,...,\alpha_n$ weights of the torus. Formula (\ref{f1}) implies that:
	\begin{multline*}
		mC^{\T}\left(Conf_k(\C^n)\right)=
		\sum_{P \in X_k}
		\prod_{P_a\in P}(-1)^{|P_a|-1} \left(|P_a|-1\right)!\left( 
		\prod_{j=1}^n\left(1+\frac{y}{\alpha_j}\right)\left(1-\frac{1}{\alpha_j}\right)^{|P_a|-1}\right)
	\end{multline*}
	 \end{ex}
 \begin{ex} \label{ex:pr}
 	Consider the projective space $\PP(\C^n)$ with action of the diagonal torus. Denote by $\alpha_1,...,\alpha_n$ weights of the torus. Denote by $\ee_1,...,\ee_n$ the fixed points of the action on $\PP(\C^n)$. Let $\ee=(\ee_{\iota_1},...,\ee_{\iota_{k}}) \in \PP(\C^n)^{k}$ be a fixed point. Such point induces partition $P_\ee$ of the set $[k]$ such that $x,y$ belong to the same element of the partition $P_\ee$ if and only if $\iota_x=\iota_y$.
 	Consider the set $X_{\ee}$ of partitions of the set $[k]$ which subdivide $P_\ee$. For an element of such partition $P_a \in P \in X_\ee$ let $\iota(P_a)$ denote the number $\iota_x$ for any $x\in P_a$.  Formula (\ref{f1}) implies that:
 	\begin{multline*}
 		mC^\T\left(Conf_k(\PP(\C^n))\right)_{|\ee}= \\
 		\sum_{P \in X_{\ee}}\prod_{P_a\in P} (-1)^{|P_a|-1}\left(|P_a|-1\right)!
 		\left(\prod_{j\neq \iota(P_a)} \left(1+\frac{y\alpha_{\iota(P_a)}}{\alpha_j}\right)\left(1-\frac{\alpha_{\iota(P_a)}}{\alpha_j}\right)^{|P_a|-1} \right)
 	\end{multline*}
 	
 \end{ex}
	
\subsection{Orbit Configuration space}
	The case of an orbit configuration space is very similar to the case of classical configuration space. For a partition $P\in X_k$ let
	$$B_P=\{(\vv_1,...,\vv_k)\in (\C^n-\{0\})^k|\lin( \vv_i) = \lin(\vv_j) \text{ when $P(i)=P(j)$}\}\,.
	$$
	The inclusion-exclusion formula once more implies the equality
	$$[C_k(\C^n)\to (\C^n)^k]=\sum_{P\in X_k}a(P)[B_P\to (\C^n)^k]\,,$$
	where numbers $a(P)$ are the same as in the previous subsection. Next step is computation of the classes $mC^{\T_\alpha\times\T_\beta}(B_P\to B^k)$. Motivic Chern classes are multiplicative with respect to the cartesian product of morphisms (\cite{AMSS} remark 4.3) in the sense that:
	$$mC^\T(f\times f': X\times X' \to Y\times Y')=mC^\T(f: X\times Y) \boxtimes
	mC^\T(f': X'\times Y')\,.$$
	So it is enough to compute the class $mC^{\T_\alpha\times\T_\beta}(B_P\to B^k)$ for the partition with only one element. 
	Namely
	$$B_k:=B_{\{[k]\}}=\{(\vv_1,...,\vv_k)\in (\C^n-\{0\})^k|\dim\lin(\vv_1,...,\vv_k) = 1\}\,.
	$$
	Unfortunately this variety doesn't have a smooth closure in the affine space $(\C^n)^k$ so we have to resolve its singularities. Consider the variety
	$$\tilde{B}_k \{(\vv_1,...,\vv_k,l)\in (\C^n-\{0\})^k\times \PP(\C^n)|\forall_i\vv_i\in l\} \subset (\C^n)^k\times \PP(\C^n), $$
	It has a smooth closure in $(\C^n)^k\times \PP(\C^n)$ (its boundary is SNC divisor) and the diagram
	\begin{align*}
	\xymatrix
	{
		\tilde{B}_k\ar[d]^{\simeq} \ar[r] & (\C^n)^k\times \PP(\C^n)^t \ar[d]^p\\
		B_k  \ar[r]& (\C^n)^k
	}
	\end{align*}
	commutes. Moreover the vertical arrow $\tilde{B}_k \to B_k$ is an isomorphism. It implies that
	$$mC^{\T_\alpha\times\T_\beta}(B_k \to (\C^n)^k)=p_*mC^{\T_\alpha\times\T_\beta}(\tilde{B}_k\to (\C^n)^k\times \PP(\C^n))\,.$$
	To compute the equivariant pushforward we use the Lefschetz-Riemann-Roch (\ref{tw:LRR})
	$$mC^{\T_\alpha\times\T_\beta}\left(B_k \to (\C^n)^k\right)=\sum_{i=1}^n \left(
	\frac{mC^{\T_\alpha\times\T_\beta}\left(\tilde{B}_k\to (\C^n)^k\times \PP(\C^n)\right)_{|{0\times\ee_i}}}{\lambda_{-1}(T^*\PP(\C^n))_{|\ee_i}}
	\right)$$
	For a fixed point $0\times\ee_i$ consider standard affine chart $U_i$ on $(\C^n)^k\times\PP(\C^n)$. The equivariant Verdier-Riemann-Roch theorem (\cite{AMSS} theorem 4.2)
	for an open embedding implies that to compute the motivic Chern class of $\tilde{B}_k$ at point $\ee_i$ one can consider only affine patch $U_i$
	$$mC^{\T_\alpha\times\T_\beta}\left(\tilde{B}_k\to (\C^n)^k\times \PP(\C^n)\right)_{|0\times\ee_i}= mC^{\T_\alpha\times\T_\beta}\left(\tilde{B}_k\cap  U_i\to U_i\right)_{|0}\,.$$
	Consider coordinates $\{x_{a,j}\}_{a\le k, j\le n}$ on $(\C^n)^k$ and projective coordinates $[y_j]_{j\le n}$ on $\PP(\C^n)$. On the set $U_i$ we have affine coordinates given by $\{x_{a,j}\}$ with torus weights $\beta_a+\alpha_j$ and $\frac{y_j}{y_i}$ with torus weight $\alpha_j-\alpha_i$.
	The variety $\tilde{B}_k \cap U_i$ is given by equations:
	$$
	\begin{cases}
	 x_{a,i} \neq 0  \text{ for all }a \in [k]\,, \\
	 x_{a,j}-x_{a,i}\frac{y_j}{y_i} = 0  \text{ for all }{a \in [k], j \in [n], j \neq i}\,.
	\end{cases}
	$$
	We choose new coordinates on the affine space $U_i$. Let variables  $\frac{y_j}{y_i}$ be the same as before and consider variables
	$$z_{a,j}=\begin{cases}
	x_{a,i} \text{ if } j=i\,,  \\
	x_{a,j}-x_{a,i}\frac{y_j}{y_i} \text{ if } j\neq i\,,
	\end{cases}$$
	with torus weight $\beta_a+\alpha_i$. Now the set $\tilde{B}_k$ is given by coordinate equations so we can use the fundamental calculation from \cite{FRW}.
	\begin{alemat}[\cite{FRW} subsection 2.7] \label{lem:FRW}
		Let a torus $\T=(\C^*)^r$ act on $\C$, and let the class of this equivariant line bundle over a point be $\alpha\in K^{\T}(pt)$. Using the additivity of the motivic Chern class we have
		\begin{equation}\label{eqn:key}
		mC^\T(\{0\}\subset \C)=1-1/\alpha,
		\qquad
		mC^\T(\C\subset \C)=1+y/\alpha,
		\end{equation}
		\[
		mC^\T(\C-\{0\}\subset \C)=(1+y/\alpha)-(1-1/\alpha)=(1+y)/\alpha.
		\]
	\end{alemat}
	\begin{rem}
		Note that the last equality comes from the additivity of $mC$,
		and the first two equalities come from the Thom isomorphism 
		\hbox{$K^{\T}(\C) \simeq K^{\T}(\{0\}).$} We have
		$$mC^{\T}(\{0\}\subset\C)=eu^\T(\{0\} \subset \C) =\lambda_{-1}(\C^{\vee}) =1-1/\alpha$$
		and
		$mC^{\T}(id_C)=\lambda_y(C^\vee) = 1 + y/\alpha.$
	\end{rem}	
	To simplify notation consider function
	$$
	\psi_{i,j}(\theta)=\begin{cases}
	1-\frac{1}{\theta} \text{ for } i\neq j \,,\\
	\frac{1+y}{\theta} \text{ for } i=j \,.\\
	\end{cases}
	$$
	Application of the lemma \ref{lem:FRW} and multiplicative property of the motivic Chern class leads to formula
	$$mC^{\T_\alpha\times\T_\beta}\left(\tilde{B}_k \cap U_i \subset U_i \right)_{|\ee_i}=\lambda_y(T^*\PP(\C^n))_{|\ee_i}\prod_{j=1}^n\prod_{a=1}^k\psi_{i,j}(\beta_a\alpha_j)\,.
	$$

	Combining all the ingredients we obtain the final formula
	\begin{align*}
	&mC\left(C_k(\C^n)\right)=\sum_{P \in X}
	\prod_{P_k\in P} \left[ (-1)^{|P_k|-1} (|P_k|-1)!
	\sum_{i=1}^n \left(
	\prod_{j\neq i} \frac{1+y\frac{\alpha_i}{\alpha_j}}{1-\frac{\alpha_i}{\alpha_j}}
	\prod_{j=1}^n\prod_{a\in P_k}\psi_{i,j}(\beta_a\alpha_j) 
	\right)\right]
	\end{align*}
	\begin{rem}
		The class $mC(C_k(\C^n))$ can be computed in alternative way using the fact that space $C_k(\C^n)$ is a multiple affine cone (without the vertex) over the configuration space $Conf_k(\PP(\C^n))$ and following the procedure given in \cite{FRWp}, section 4.4.
	\end{rem}
	
\section{Generating series} \label{s:szer}
It is a common practice to present calculation of genera, or characteristic classes of family of varieties in the form of a generating series (e.g.
 \cite{Oh,CMSSY,Bal} for symmetric products,
 \cite{BoNi} for Hilbert schemes, \cite{CMOSSY} for
 unordered configuration spaces and 
 \cite{MS2}
 for products with
 their symmetric group actions). In this section 
we aim to prove the equalities in a
power series $R[[t]]$ for $R$ a suitable $\Q$-algebra, using the power series
$$
\exp(t)=\sum_{k=0}^{\infty}\frac{1}{k!}\cdot t^k \text{ and } \log(1+t)= \sum_{k=1}^{\infty}\frac{(-1)^{k-1}}{k}\cdot t^k.
$$

 In the following we consider $\T$-equivariant closed embedding
$B \subset M$ of complex quasiprojective varieties with $M$
smooth.
\begin{atw}
	\label{s3}
	\begin{align*}
	&1+\sum_{k=1}^{\infty}\frac{t^k}{k!}\cdot
	mC^{\T}\left(Conf_k(B) \subset M^k\right)_{|M}=
	\exp\left(mC^\T(B\subset M)
	\frac{\log\left(1+t\cdot eu^\T(TM)\right)}{eu^\T(TM)}
	\right)
	\end{align*}
	in the power series ring over $R = G^\T(M)[y]\otimes\Q,$ with the
	multiplication coming from the tensor product
	on $G^\T(M) = K^\T(M)$ and the restrictions $|M$ with respect to the diagonal embeddings
	$d_k: M \subset M^k.$ Here we use the formal power series
	$$\frac{log(1+at)}{a}=\sum_{k=1}^\infty \frac{(-1)^{k-1}a^{k-1}}{k!} \cdot t^k $$
\end{atw}
Assume now that $pt \in M$ is an isolated $\T$-fixed point in $M$.
\begin{atw}
	\label{s1}
	\begin{align*}
	&1+\sum_{k=1}^{\infty}\frac{t^k}{k!}\cdot
	\frac{mC^{\T}\left(Conf_k(B) \subset M^k\right)_{|pt^k}}{eu^{\T}\left(\{pt^k\} \subset M^k\right)}=
	\exp\left(\frac{mC^\T(B\subset M)_{|pt}}{eu^\T(\{pt\}\subset M)}\log(1+t)\right)
	\end{align*}
	in the power series ring over $R=\left(S^{-1}K^\T(pt^k)[y]\right)\otimes \Q$.
\end{atw}
\begin{rem}
	Note that the ring $S^{-1}K^\T(pt^k)[y]$ has no $\Z$-torsion so the homomorphism 
	$$S^{-1}K^\T(pt^k)[y] \to \left(S^{-1}K^\T(pt^k)[y]\right)\otimes \Q$$
	is injective.
\end{rem}
\begin{rem}
	Similar formulae also hold with the same proof for $\T$-equivariant MacPherson Chern classes $c_*^\T$
	and Hirzebruch classes $T^\T_{y*}$ in
	equivariant (even degree) (co)homology or Chow groups, with the equivariant
	Euler class given by the corresponding top-dimensional equivariant
	Chern class of the normal bundle.
\end{rem}
Finally we have the following (exponential) generating series for the "localized" \hbox{$\T$-equivariant} motivic Chern classes of the orbit configuration
space $C_k(\C^n)$ of pairwise linearly independent vectors in $\C^n$:
\begin{atw}
	\label{s2}
	\begin{multline*}
	1+\sum_{k=1}^{\infty}\frac{t^k}{k!}\cdot
	\frac{mC^{\T_\alpha}\big(C_k(\C^n) \subset (\C^n)^k\big)}{eu^{\T_\alpha}\big(\{0\} \subset (\C^n)^k\big)} =
	\prod_{i=1}^{n}\exp\left( 
	\frac{\lambda_y(T^*\PP(\C^n))_{|\ee_i}}{\lambda_{-1}(T^*\PP(\C^n))_{|\ee_i}}
	\log\left(1+\frac{t(1+y)}{\alpha_i-1}\right)\right)
	\end{multline*}
\end{atw}
Where the functions $exp(-)$ and $log(1+(-))$ are defined as appropriate power series.
	\begin{rem}
	The formula from remark \ref{rem:conf} implies that
	$$\frac{mC^{\T_\alpha}\big(F_{\C^*}(\C^n,k))}{eu^{\T_\alpha}\big(\{0\} \subset (\C^n)^k\big)}=
	\frac{mC^{\T_\alpha}\big(C_k(\C^n) \big)}{eu^{\T_\alpha}\big(\{0\} \subset (\C^n)^k\big)}+
	k\frac{mC^{\T_\alpha}\big(C_{k-1}(\C^n) )}{eu^{\T_\alpha}\big(\{0\} \subset (\C^n)^{k-1}\big)}\,. $$
	If we denote by $f$ the power series from the theorem \ref{s2} then
	$$
	1+\sum_{k=1}^{\infty}\frac{t^k}{k!}\cdot
	\frac{mC^{\T_\alpha}\big(F_{\C^*}(\C^n,k))}{eu^{\T_\alpha}\big(\{0\} \subset (\C^n)^k\big)}=f(t)+tf'(t)\,.
	$$
	\end{rem}
Define an ordered partition of the set $[k]$ as a sequence of nonempty pairwise disjoint subsets of~$[k]$ whose sum is the whole set $[k]$. Let $Y_k$ be the set of ordered partitions of~$[k]$. For a partition $P\in Y_k$ we use notation $P=[P_1,...,P_{|P|}]$ and $p_a=|P_a|$.
\begin{ex}
	$|Y_1|=1, \ |Y_2|=3, \ |Y_3|=13,$
	$ Y_2= \{(\{1,2\}),(\{1\},\{2\}),(\{2\},\{1\})\} \,.$
\end{ex}
\begin{alemat}
	\label{lem:szeregi}
	For an $\Q$-algebra $R$ and an arbitrary function $f:\N \to R[[t]]$ 
	$$
	1+\sum_{k=1}^{\infty} \sum_{P \in X_k}
	\frac{a(P)}{k!}\prod_{P_a\in P}f(|P_a|)= \exp\left(\sum_{u=1}^{\infty}\frac{(-1)^{u-1}f(u)}{u}\right)\,,
	$$
	when both series are well defined. 
	 If we assume that $t^n|f(n)$ for all natural numbers $n$, then both series are well defined.
\end{alemat}
\begin{proof}
	Observe that
	$$ 1+\sum_{k=1}^{\infty} \sum_{P \in X_k}
	\frac{a(P)}{k!}\prod_{P_a\in P}f(|P_a|)= 1+\sum_{k=1}^{\infty} \sum_{P \in Y_k}
	\frac{a(P)}{k!|P|!}\prod_{a=1}^{|P|}f(p_a).$$
	In the following computations we start by using the fact that terms of this series depend only on cardinality of elements of a partition and not on a~partition itself. Then we introduce sum over variable $s$ corresponding to the length of the partition.
	\begin{align*}
	&1+\sum_{k=1}^{\infty} \sum_{P \in Y_k}
	\frac{a(P)}{k!|P|!}\prod_{a=1}^{|P|}f(p_a)= \\
	&1+\sum_{k=1}^{\infty} \sum_{P \in Y_k}
	\frac{1}{|P|!\binom{k}{p_1,...,p_{|P|}}}\prod_{a=1}^{|P|}\frac{(-1)^{p_a-1}f(p_a)}{p_a}= \\
	&1+\sum_{k=1}^{\infty} \sum_{s=1}^{\infty} \sum_{p_1,...,p_s>0, \sum p_j=k}
	\frac{1}{s!}\prod_{a=1}^{s}\frac{(-1)^{p_a-1}f(p_a)}{p_a}= \\
	&1+ \sum_{s=1}^{\infty}\frac{1}{s!} \sum_{p_1,...,p_s>0}
	\prod_{a=1}^{s}\frac{(-1)^{p_a-1}f(p_a)}{p_a}= \\
	&1+ \sum_{s=1}^{\infty}\frac{1}{s!} \Big(\sum_{u=1}^{\infty}\frac{(-1)^{u-1}f(u)}{u}\Big)^s=\exp\Big(\sum_{u=1}^{\infty}\frac{(-1)^{u-1}f(u)}{u}\Big)
	\end{align*}
\end{proof}
\begin{rem}
	Lemma \ref{lem:szeregi} could  be rewritten in the language of 
	number partitions.
	Namely a partition $P\in X_k$ induces a number partition
	$\lambda=(n_1,...,n_k)\dashv k$
	of $k =
	1 \cdot n_1 + ... + k \cdot n_k$
	such that $n_i$ is the number of
	elements in $P$ of size $i$. Note that
	$l(\lambda) = n_1 + ... + n_k = |P|.$ The
	number of set-partitions $P\in X_k$ with $n_i$ blocks of size $i$ for $i = 1,...,k$
	is given by (see e.g. \cite{Stan}, Equation (3.36) on p.279):
	$$\frac{k!}{\prod_{i=1}^k ((i!)^{n_i} \cdot n_i!)} \,.
	$$
	Denote by $z_\lambda=\prod_{i=1}^k(i^{n_i}\cdot n_i!)$. The formula from lemma \ref{lem:szeregi} takes form
$$
1+\sum_{k=1}^{\infty} \sum_{\lambda\dashv k} (-1)^{k-l(\lambda)} \cdot
\frac{1}{z_\lambda}\cdot
\prod_{i=1}^kf(i)^{n_i}= \exp\left(\sum_{u=1}^{\infty}\frac{(-1)^{u-1}f(u)}{u}\right)\,,
$$
	fitting with \cite{Macd} equation (2.14') and \cite{Stan} theorem 1.3.3.
\end{rem}
\begin{cor}
For the function
$f : \N \to G^\T(M)[y] \otimes Q[[t]]$
given by
$$f(u) = t^u \cdot eu^\T(TM)^{u-1} \cdot mC^\T(B \subset M)$$
we obtain a proof of theorem \ref{s3}.
\end{cor}
\begin{proof}
	We use the fact that 
	\begin{align*}&mC^\T(B^{|P|}\to M^k)_{|M}=
	\left(i_{P*}mC^\T(B^{|P|}\to M^{|P|})\right)_{|M}=
	\left(i_P^*i_{P*}mC^\T(B^{|P|}\to M^{|P|})\right)_{|M}= \\
	&=mC^\T(B^{|P|}\to M^{|P|})_{|M} \cdot eu^\T(\nu(i_P))_{|M}=mC^\T(B\subset M)^{|P|}\cdot eu^\T(TM)^{k-|P|}
	\end{align*}
	where $\nu(i_P)$ denotes normal bundle to the embedding $i_P$ and ${|M}$ restriction to diagonal.
\end{proof}

\begin{cor}
	For the function $f:\N \to \left(S^{-1}K^\T(pt)[y]\right)\otimes\Q[[t]]$ given by
	$$f(u)=t^u\cdot
	\frac{mC^\T(B\subset M)_{|pt}}{eu^\T(\{pt\}\subset M)}$$
	we obtain the proof of theorem \ref{s1}.
\end{cor}
\begin{proof}
	Note that
	$$eu^\T(\{pt^u\}\subset M^u)=eu^\T(\{pt\}\subset M)^u=eu^\T(TM)_{|pt}^u.$$
	Moreover
	\begin{align*}
	&\frac{mC^\T(B^{|P|}\to M^k)_{|pt^k}}{eu^\T(\{pt^k\}\subset M^k)}=
	\frac{\left(i_P^*i_{P*}mC^\T(B^{|P|}\to M^{|P|})\right)_{|pt^{|P|}}}{eu^\T(\{pt\}\subset M)^k}=\\
	&=
	\frac{mC^\T(B^{|P|}\to M^{|P|})_{|pt^{|P|}}\cdot eu^\T(TM)^{k-|P|}_{|pt}}{eu^\T(\{pt\}\subset M)^k}=
	\frac{mC^\T(B\to M)^{|P|}_{|pt}}{eu^\T(\{pt\}\subset M)^{|P|}}
	\end{align*}
\end{proof}
\begin{cor}
	\label{cor:szer}
	For the function $f:\N \to S^{-1}K^\T(pt)[y][[t]]$ given by
	$$f(u)=\sum_{i=1}^n \bigg[
	\bigg(\frac{(1+y)t}{(\alpha_i-1)}\bigg)^u
	\prod_{j\neq i}
	\frac{1+y\frac{\alpha_i}{\alpha_j}}{1-\frac{\alpha_i}{\alpha_j}}
	\bigg]\,,$$
	we obtain the proof of theorem \ref{s2}
\end{cor}
\begin{proof}
	\begin{multline*}
	1+\sum_{k=1}^{\infty}\frac{1}{k!}\cdot	\frac{mC^{\T_\alpha}\big(C_k(\C^n) \subset (\C^n)^k\big)}{eu^{\T_\alpha}\big(\{0\} \subset (\C^n)^k\big)}\cdot t^k= \\
	=\exp\left(\sum_{u=1}^{\infty}t^u(1+y)^u\frac{(-1)^{u-1}}{u}\cdot
	\left(\sum_{i=1}^n 
	\frac{1}{(\alpha_i-1)^u}
	\prod_{j\neq i}
	\frac{1+y\frac{\alpha_i}{\alpha_j}}{1-\frac{\alpha_i}{\alpha_j}}
	\right)\right) = \\
	=\exp\left(\sum_{i=1}^n\left(\prod_{j\neq i}
	\frac{1+y\frac{\alpha_i}{\alpha_j}}{1-\frac{\alpha_i}{\alpha_j}}\right)
	\left(\sum_{u=1}^{\infty}\frac{(-1)^{u-1}}{u}\left(\frac{t(1+y)}{\alpha_i-1}\right)^u
	\right)\right)= \\
	=\prod_{i=1}^{n}\exp\left(\log\left(1+\frac{t(1+y)}{\alpha_i-1}\right)
	\prod_{j\neq i}
	\frac{1+y\frac{\alpha_i}{\alpha_j}}{1-\frac{\alpha_i}{\alpha_j}}\right)
	\end{multline*}
\end{proof}

The expression in theorem \ref{s2} can be written in a different form using residues methods  (See \cite{Zie}
 for cohomology and \cite{WeZie}
  for $K$-theory).
Consider the function
$$F(z)=\frac{1}{z(1+y)}
\log\left(1+\frac{t(1+y)}{z-1}\right)
\prod_{i=1}^n
\frac{1+y\frac{z}{\alpha_i}}{1-\frac{z}{\alpha_i}} \,.
$$
\begin{rem}
	The logarithm inside of the function $F$ is a formal power series, not a branch of the complex logarithm. We take residues in each gradation of the formal variable $t$	separately.		
\end{rem}
The function $F$ has residues in 
$z\in\{\alpha_1,...,\alpha_n,1,0,\infty\}$. Residues in $0$ and $\infty$ are easily computable:
$$Res_{z=0}F(z)=\frac{\log(1-t(1+y))}{(1+y)}\text{  ,  } Res_{z=\infty}F(z)=
0\,.$$
So the residue theorem implies that
\begin{multline*}
1+\sum_{k=1}^{\infty}\frac{1}{k!}\cdot
\frac{mC^{\T_\alpha}\big(C_k(\C^n) \subset (\C^n)^k\big)}{eu^{\T_\alpha}\big(\{0\} \subset (\C^n)^k\big)}\cdot t^k= \\
=
\exp\left(-\sum_{i=1}^{n}Res_{z=\alpha_i}F(z)
\right)=
\exp\left(Res_{z=0}F(z)+Res_{z=1}F(z)+Res_{z=\infty}F(z)
\right)= \\
=\exp\left(
\frac{\log(1-t(1+y))}{(1+y)}+Res_{z=1}\left(
\frac{\lambda_{yz}(T^*\C^n)}{z(1+y)\lambda_{-z}(T^*\C^n)}
\cdot\log\left(1+\frac{t(1+y)}{z-1}\right)\right)\right)= \\
=\left(1-t(1+y)\right)^{\frac{1}{1+y}}Res_{z=1}\left(
\frac{\lambda_{yz}(T^*\C^n)}{z(1+y)\lambda_{-z}(T^*\C^n)}
\cdot\log\left(1+\frac{t(1+y)}{z-1}\right)\right)\,.
\end{multline*}
	
	\section{Stability} \label{s:BB}
	\subsection{Change of number of points}
	For any variety $B$ there is a decomposition in the group $K^\T(Var/B^k)$
	$$[Conf_{k+1}(B) \to B^{k+1}]=[Conf_{k}(B)\times B \to B^{k+1}] - \sum_{i=1}^k [Conf_{k}(B)\to B^k\xto{j_i} B^{k+1}]\,,$$
	where the map $j_i$ is given by
	$j_i(x_1,...,x_k)=(x_1,...,x_k,x_i).$
	The maps $j_i$ are proper, so this decomposition induces an equality
	$$mC^\T\big(Conf_{k+1}(B)\big)=
	mC^\T\big(Conf_k(B)\big)\boxtimes 	mC^\T\big(B\big)-\sum\limits_{i=1}^k j_{i*}mC^\T\big(Conf_k(B)\big)\,.$$
	\begin{ex}
		Consider the projective space $\PP(\C^n)$ with the natural action of the diagonal torus $\T_\alpha$. Denote the fixed points by $\ee_1,...,\ee_n$ of the action on $\PP(\C^n)$. Choose a fixed point \hbox{$\ee=(\ee_{\iota_1},...,\ee_{\iota_{k+1}}) \in \PP(\C^n)^{k+1}$} and let
		$$I=\{j\le k| \iota_j=\iota_{k+1}\}\,.$$
		Let  $\tilde{\ee}=(\ee_{\iota_1},...,\ee_{\iota_{k}})$. Then
		\begin{align*}
		&mC^\T\big(Conf_{k+1}(\PP(\C^n))\big)_{|\ee}=\\
		&=mC^\T\big(Conf_k(\PP(\C^n))\big)_{|\tilde{\ee}} 	\lambda_y\big(\PP(\C^n)\big)_{|\ee_{\iota_{k+1}}}-\sum_{i\in I} mC^\T\big(Conf_k(\PP(\C^n))\big)_{|\tilde{\ee}}eu(\nu(j_i))=\\
		&=mC^\T\big(Conf_k(\PP(\C^n))\big)_{|\tilde{\ee}}
		\left(
		\lambda_y\big(\PP(\C^n)\big)_{|\ee_{\iota_{k+1}}}-
		|I|\left(eu \left(\nu(\ee_{\iota_{k+1}} \subset \PP(\C^n))\right)\right)\right)= \\
		&=mC^\T\big(Conf_k(\PP(\C^n))\big)_{|\tilde{\ee}}
		\left(
		\prod_{i\neq \iota_{k+1}}\left(1+y\frac{\alpha_{\iota_{k+1}}}{\alpha_{i}}\right)-
		|I|\prod_{i\neq \iota_{k+1}}\left(1-\frac{\alpha_{\iota_{k+1}}}{\alpha_{i}}\right)
		\right)\,.
		\end{align*}
	\end{ex}
	\subsection{Connections with BB decomposition}
	When one dimensional torus $\C^*$ acts on a smooth, projective variety $M$, there is the BB-decomposition \cite{B-B1,B-B2,B-B3} (see also \cite{Bro})
	$$M= \bigsqcup_{F\in M^{\C^*}} B_F^+\,,$$
	where
	$$B_F^+=\{x\in M|\lim_{t \to 0}tx \in F\}\,.$$
	Such decomposition induces a partial order on the fixed points components such that $ F_1 \ge F_2$ if and only if the closure of the BB cell $B^+_{F_1}$ intersects $F_2$. We call the minimal and maximal elements of such decomposition sink and source respectively.
\\
		Assume that a torus $\T$ with a chosen one dimensional subtorus $\C^* \subset \T$ acts on a smooth, projective variety $M$. Consider the BB decomposition
		for the action of the one dimensional subtorus $\C^*$. Let $F_1 \subset M^{\C^*}$ be the sink (or the source) of this decomposition. It turns out that the motivic Chern class of the configuration space of $M$ determines the class of the configuration space of $F_1$.
		To formulate this fact formally we define a limit map.
		\begin{adf} \label{df:lim}
			 Let $F$ be a $\T$ variety with the trivial action of a chosen subtorus $\C^* \subset \T$. Consider the quotient torus $\T_1=\T/\C^*$. Denote by ${\bf t}$ the character of torus $\C^*$. Choose a split $\T_1 \times \C^* \simeq \T$. Consider a subring 
			$$
			K^{\T_1}(F)[t] \subset K^{\T_1}(F)[t,t^{-1}] \simeq K^{\T_1\times \C^*}(F)\simeq K^\T(X) \,,
			$$
			 where the first isomorphism follow from the paragraph 5.2.1 \cite{ChGi}, and the second is given by the split. Define the map $\lim_{{\bf t} \to 0}$
			 \begin{align*}
			 	K^{\T_1}(F)[{\bf t}] \to K^{\T_1}(F)
			 \end{align*}	
			by killing all positive powers of ${\bf t}$.
			 The subring $K^{\T_1}(F)[{\bf t}] \subset K^\T(F)$ and the map $\lim_{{\bf t} \to 0}$ are independent from the choice of splitting. \\
			Consider a subring $A$ of the localized $K$-theory defined by
			$$
			A=\{f\in S^{-1}K^{\T}(F)|f=\frac{a}{b}, a\in K^{\T_1}(F)[{\bf t}], b\in K^{\T_1}(pt)[{\bf t}], b\notin {\bf t}K^{\T_1}(pt)[{\bf t}]\}\,.
			$$
			The map $\lim\limits_{{\bf t} \to 0}$ extends to this ring by applying limit map to the numerator and the denominator separately.
			Analogously one can define the limit map $\lim\limits_{{\bf t}^{-1} \to 0}$.

		\end{adf}
	\begin{alemat}
		The limit map is well defined.
	\end{alemat}
	\begin{proof}
		We need to show that the limit map is independent from the choice of split and representation as fraction.
		Consider two split maps $s_1,s_2: \T_1 \to \T$. Let $\alpha_1,\alpha_2$ be isomorphisms
		$$K^{\T_1}(F)[{\bf t},{\bf t}^{-1}] \to K^{\T}(F)$$
		induced by these maps. Note that the quotient $\frac{s_2}{s_1}$ induce a group homomorphism \hbox{$h:\T_1 \to \C^*$.} Direct calculation provide us with the formula 
		$$\alpha_2^{-1}\alpha_1(E{\bf t}^i)=E\otimes\C_{h^i}{\bf t}^i$$
		for $E \in K^{\T_1}(F)$. It implies independence of the limit map from a split. \\
		Now we need to check that for $\alpha\in A$ the element $\lim\limits_{{\bf t} \to 0} \alpha$ doesn't depend on a choice of representation of $\alpha$ as a fraction. Assume that we have two  representations $\alpha=\frac{a}{b}=\frac{c}{d}$ satisfying conditions from the definition of limit map. Namely
		\begin{align*}
		a=\sum_{i \ge 0} a_i{\bf t}^i, c=\sum_{i \ge 0} c_i{\bf t}^i, b=\sum_{i \ge 0} b_i{\bf t}^i, d=\sum_{i \ge 0} d_i{\bf t}^i,
		\end{align*}
		where $b_0$ and $d_0$ are nonzero and all sums are finite. Our aim is to prove that
		$$
		\frac{a_0}{b_0}=\frac{c_0}{d_0}\,.
		$$
		The equality $\frac{a}{b}=\frac{c}{d}$ implies that there exist a nonzero element
		$$s=\sum_{i \in \Z} s_i{\bf t}^i \in K^\T(pt) \,,$$
		such that
		$scb=sad$. Let $i_{min}$ denote the smallest number $i$ such that $s_i \neq 0$. Then looking at coefficient ${\bf t}^{i_{min}}$ we acquire equation
		$$s_{i_{min}}c_0b_0=s_{i_{min}}a_0d_0\,.$$
		Which proves the lemma.
	\end{proof}
	\begin{rem}
		Limit map can be defined geometrically in a split-free way.
		Consider an algebraic tours $\T$ with a chosen subtorus $\C^* \subset \T$ and the quotient torus $\T_1=\T/\C^*$. Choose a partial completion $\C^* \subset \C$. Interpret $R(\T)$ as subring of functions on algebraic torus $\T$. We have inclusions
		$$\T \subset \T\times_{\C^*} \C \supset \T\times_{\C^*} \{0\} \simeq \T_1\,.$$
		The limit map on $R(\T)$ is defined on the functions that extend to variety $\T\times_{\C^*} \C$ as a restriction to the zero fiber (which is equal to $\T_1$).
		Now assume that $F$ is a $\T$-variety with the trivial action of $\C^*$. There is natural morphism
		$$K^{\T_1}(F) \otimes_{R(\T_1)}R(\T) \to K^\T(F)\,.$$
		Any choice of a split $\T \simeq \T_1 \times \C^*$  imply that this map is an isomorphism. The limit map is defined by applying the limit map to $R(\T)$ (on a subring where its defined). To prove that this definition coincides with the one given in \ref{df:lim}, note that the choice of split induces isomorphism $\T\times_{\C^*} \C \simeq \T_1 \times \C$.
	\end{rem}
		
		Let's note some basic properties of the limit map.
	\begin{pro}
		\begin{enumerate}
			\label{pro:lim}
			\item The limit map is additive and multiplicative.
			\item It commutes with proper pushforwards (by maps between varieties with trivial $\C^*$ action).
			\item It commutes with pullbacks (by maps between varieties with trivial $\C^*$ action)
		\end{enumerate}
	\end{pro}
	\begin{proof}
		First point is obvious. Second and third follows from the fact that for any \hbox{$\T$ varieties} $X,Y$ with trivial $\C^*$ action and $\T$-equivariant map $f$ between them there are isomorphisms of cohomology rings:
		$$K^{\T}(X)=K^{\T_1}(X)[{\bf t},{\bf t}^{-1}], \ K^{\T}(Y)=K^{\T_1}(Y)[{\bf t},{\bf t}^{-1}]$$
		and the maps $f^{\T}_*,f_{\T}^*$ are given by applying $f^{\T_1}_*,f^*_{\T_1}$ to the coefficients of Laurent polynomials.
	\end{proof}
	Now we can formulate the theorem.
		\begin{atw}
			\label{twABB}
			Assume that a torus $\T$ acts on a smooth compact variety $M$. Consider one dimensional subtorus $\C^* \subset \T$. Denote by $\T_1$ the quotient torus $\T/\C^*$. Let $F$ be a component of the fixed points of \hbox{the torus $\C^*$}. Let $M_F^+$ be the positive BB cell of $F$. Then for any $\T$-equivariant map $f: X \to M$ the limit map is well defined on the element
			$$ \frac{mC^{\T}(f:X\to M)_{|F}}{eu^{\T}(\nu_F)} \in S^{-1}K^T(F)\,. $$
			Moreover
			\begin{align}
			\label{f2}
			mC^{\T_1}(f_F:f^{-1}(M_F^+)\to F)=\lim_{{\bf t} \to 0}
			\left(\frac{mC^{\T}(f:X\to M)_{|F}}{eu^{\T}(\nu_F)}\right)\,.
			\end{align}
			Analogous result is true for negative BB cells and the limit map $\lim\limits_{{\bf t}^{-1} \to 0}$.
		\end{atw}
	\begin{rem}
		The right hand side of formula (\ref{f2}) lives in the localized $K$-theory of a fixed points component. The Euler class $eu^{\T}(\nu_F)$ is invertible in this ring (lemma \ref{lem:eu}).
	\end{rem}
	This theorem is generalization of the theorem 10 from \cite{WeBB}.
	The proof is analogous but some formal modifications need to be done. It is the content of the next section.
	\begin{rem}
		In our case the limit map is simpler than the one in \cite{WeBB}. It is consequence of using the motivic Chern class in $K$-theory instead of the Hirzebruch class in cohomology. The equivariant Hirzebruch class lives in completed equivariant cohomology ring, so one needs to define limit of power series instead of Laurent polynomial.
	\end{rem}
	\begin{cor}
		\label{corst}
		Let $F_1$ be the sink of the BB decomposition, then
		$$
		mC^{\T_1}\left(Conf_k(F_1)\to F_1^k\right)=\lim_{{\bf t} \to 0}
		\left(mC^{\T}\left(Conf_k(M)\to M^k\right)_{|F_1^k}\right) \,.
		$$
	Analogues result hold when $F_1$ is the source.
	\end{cor}
	\begin{proof}
		In the theorem \ref{twABB} choose $M^k$ as a smooth variety. Let $f$ be the inclusion of configuration space and  $F:=F_1^k$ product of the sinks
		in the BB decomposition of $M$. Then the cell $B_F^+$
		is equal to $F$. Moreover the action of one dimensional torus $\C^*$ on the normal bundle $v_F$ has only positive weights.
		It follows that
		$$\lim_{{\bf t} \to 0} eu^\T(\nu_{F_1^k})=1\,.$$
		Now the corollary follows from the theorem \ref{twABB}.
	\end{proof}
	\begin{ex}
		Consider the projective space $\PP(\C^n)$ with action of the diagonal torus 
		$$\T_n=\T_{n-1} \times\C^*\alpha_n \,,$$
		where the subtorus $\T_{n-1}$ acts on the first $n-1$ coordinates, and $\C^*\alpha_n$ on the last coordinate. Action of $\C^*$ has two fixed points components: the source is a projective space $\PP(\C^{n-1})$ with natural action of $\T_{n-1}$ and the sink is a point. Choose a fixed point
		$$\ee \in \PP(\C^{n-1}) \subset \PP(\C^{n})\,.$$
		The direct computations (example \ref{ex:pr})
		shows that the class $mC^{\T_{n}}\big(Conf_{k}(\PP(\C^{n}))\big)_{|\ee}$ belongs to subring $\Z[\alpha_1^\pm,...,\alpha_{n-1}^\pm,\alpha_{n}^{-1}]
		 \subset K^{\T_{n}}(\ee)$,
		moreover
		\begin{align*}
		\lim_{\alpha_n^{-1} \to 0}\left(mC^{\T_{n}}\big(Conf_{k}(\PP(\C^{n}))\big)_{|\ee}\right)
		=mC^{\T_{n-1}}\big(Conf_{k}(\PP(\C^{n-1}))\big)_{|\ee}\,.
		\end{align*}
		As predicted by corollary \ref{corst}.
	\end{ex}
	
	\section{Proof of theorem \ref{twABB}}
	First step of the proof is reduction to the case when the map $f$ is identity on a smooth variety. Both sides of the formula (\ref{f2}) are additive with respect to addition in the Grothendieck group of varieties $K^{\T}(Var/M)$. Proper maps from smooth varieties generate Grothendieck group
	(cf. section 5, \cite{WeBB})
	so we can assume that $f$ is proper and $X$ is smooth.
	Assume that the theorem \ref{twABB} holds for $id_M$. Then for any proper map from a smooth variety:
	\begin{multline*}
	\lim_{{\bf t} \to 0}\left(\frac{mC^{\T}(f:X\to M)_{|F}}{eu^{\T}(\nu_F)}\right)=
	\lim_{{\bf t} \to 0}\left(\frac{f_*mC^{\T}(id_X)_{|F}}{eu^{\T}(\nu_F)}\right)= \\
	=\lim_{{\bf t} \to 0}\left(\sum_{G\subset X^{\C^*}\cap f^{-1}(F)}f^{|G}_*\frac{mC^{\T}(id_X)_{|G}}{eu^{\T}(\nu_G)}\right)=	\sum_{G\subset X^{\C^*}\cap f^{-1}(F)}f^{|G}_*\lim_{{\bf t} \to 0}\left(\frac{mC^{\T}(id_X)_{|G}}{eu^{\T}(\nu_G)}\right) = \\
	=\sum_{G\subset X^{\C^*}\cap f^{-1}(F)}f^{|G}_*
	mC^{\T_1}\left(X_G^+\to G\right)=\sum_{G\subset X^{\C^*}\cap f^{-1}(F)}
	mC^{\T_1}\left(X_G^+\to F\right) =\\
	= mC^{\T_1}\left(f^{-1}(M_F^+)\to F\right)
	\end{multline*}
	We have used the relative version of Lefschetz-Riemann-Roch (theorem \ref{twLRR}), commutation of the limit map with pushforwards and additivity of the limit map (proposition \ref{pro:lim}). \\ 
	To finalize we prove the formula (\ref{f2}) for $f=id_M$. The Variety $M$ is smooth so there is an equality 
	$$\frac{mC^{\T}(id_M)_{|F}}{eu^{\T}(\nu_F)}=\lambda^{\T}_y(T^*F)\frac{\lambda^{\T}_{y}(\nu^*_F)}{\lambda^{\T}_{-1}(\nu^*_F)}\,.$$
	The action of $\C^*$ on the vector bundle $T^*F$ is trivial thus
	$$\lim_{{\bf t}\to 0}\lambda^{\T}_y(T^*F)=\lambda^{\T_1}_y(T^*F)\,.$$
	\begin{alemat}[cf. corollary 19 from \cite{WeBB}
		]
		Let $n^+$ be the number of positive weights of the torus $\C^*$ on the vector bundle $v_F$, then
		$$\lim_{{\bf t}\to 0}\frac{\lambda^{\T}_{y}(\nu^*_F)}{\lambda^{\T}_{-1}(\nu^*_F)}= (-y)^{n^+}\,.$$
	\end{alemat}
\begin{proof}
	All weights of the torus $\C^*$ on the vector bundle $\nu_F$ are nonzero because it is a normal bundle to the fixed point component. Consider the weight decomposition of $\nu_F^*$ according to the torus $\C^*$
	$$\nu_F^*= \bigoplus_{\omega\in \Z,\omega \neq 0} E_\omega \otimes \C^{\omega{\bf t}}\,,$$
	where $E_\omega$ are bundles over $F$ with the trivial $\C^*$ action. For $\omega>0$
	$$\lim_{{\bf t}\to 0}\frac{\lambda^{\T}_{y}(E_\omega \otimes \C^{\omega{\bf t}})}{\lambda^{\T}_{-1}(E_\omega \otimes \C^{\omega{\bf t}})}=
	\lim_{{\bf t}\to 0}\frac{\sum_k y^k{\bf t}^{k\omega}\Lambda^kE_\omega } {\sum_k (-1)^k{\bf t}^{k\omega} \Lambda^kE_\omega}=1 \,.
	$$
	On the other hand for $\omega<0$
	$$
	\lim_{{\bf t}\to 0}\frac{\sum_k y^k{\bf t}^{k\omega}\Lambda^kE_\omega } {\sum_k (-1)^k{\bf t}^{k\omega} \Lambda^kE_\omega}=
	\lim_{{\bf t}\to 0}\frac{\sum_k y^k{\bf t}^{(k-\dim(E_\omega))\omega}\Lambda^kE_\omega } {\sum_k (-1)^k{\bf t}^{(k-\dim(E_\omega))\omega} \Lambda^kE_\omega}= \frac{y^{\dim(E_\omega)} \det(E_\omega)}{(-1)^{\dim(E_\omega)} \det(E_\omega)}=(-y)^{\dim(E_\omega)}\,.
	$$
	It follows that
	$$\lim_{{\bf t}\to 0}\frac{\lambda^{\T}_{y}(\nu^*_F)}{\lambda^{\T}_{-1}(\nu^*_F)}= (-y)^{\sum\limits_{\omega<0} \dim(E_\omega)}= (-y)^{n^+}\,.$$
\end{proof}
	Thus the right hand side of the formula (\ref{f2}) for $f=id_M$ is equal to $$\lambda^{\T_1}_y(T^*F)(-y)^{dim M_F^+}=mC^{\T_1}(id_F)(-y)^{dim M_F^+}\,.$$
 But $M_F^+$ is a Zariski affine bundle over $F$ (theorem 4.3 \cite{B-B1})
 so $$mC^{\T_1}(id_F)(-y)^{dim M_F^+}=mC^{\T_1}( M_F^+ \to F)\,.$$ Which proves the theorem \ref{twABB}.
 
 \section{Appendix 1: relative localization theorems}
\begin{alemat}
	\label{lem:eu}
	Assume that a torus $\T$ acts on a smooth variety $X$. Consider a subtorus $\T_1 \subset \T$. Let $F \subset X^{\T_1}$ be a component of the fixed points set of the chosen subtorus. Then the Euler class of the normal bundle $eu(\nu_F)$ is invertible in localized $K$-theory $S^{-1}K^\T(X)$.
\end{alemat}
\begin{proof}
	The lemma follows from slightly more general fact. Namely for a smooth $\T$-variety $Y$ element $a \in S^{-1}K^\T(Y)$ is invertible if and only if it is invertible after restriction to each fixed points component. \\
	The "only if" part is trivial. Let's prove the "if" part. Assume that for every component $H\subset X^\T$ element $i^*_H a$ has inverse $a_H$. Consider the element $$\tilde{a}=\sum_{H \subset X^{\T}} i_*^H\frac{a_H}{eu^\T(\nu_H)} \in S^{-1}K^\T(Y)\,.$$
	Then for every component $G\subset X^\T$ restriction
	$$i^*_G(a\tilde{a})=i^*_G(a)\sum_{H \subset X^{\T}} i^*_Gi_*^H\frac{a_H}{eu^\T(\nu_H)}=i^*_G(a)a_G=1\,.$$
	Then the localization theorem (\cite{Tho}, theorem 2.1) implies that $a\tilde{a}=1$. \\
	To complete the proof of the lemma consider the element $eu(\nu_F) \in S^{-1}K^\T(F)$.
	Bundle $\nu_F$ restricted to the fixed points set has all weights nontrivial. It is a classical fact that such bundle has invertible Euler class at fixed points (proposition 5.10.3 \cite{ChGi}). Thus $eu(\nu_F)$ is invertible in $S^{-1}K^\T(F).$
\end{proof}
 \begin{atw}[relative first localization theorem cf. \cite{Tho} theorem 2.1, or \cite{Se} proposition 4.1 for topological case] 
 	Assume that a torus $\T$ acts on a smooth variety $X$. Let $X^{\T_1}$ denote fixed points of a subtorus $\T_1 \subset \T$. Then pullback map in localized $K$-theory
 	$$
 	S^{-1}K^\T(X)\to S^{-1}K^\T(X^{\T_1})
 	$$
 	is an isomorphism.
 \end{atw}
 \begin{proof}
 	Observe that the diagram
 	$$
 	\xymatrix{
 	S^{-1}K^\T(X) \ar[d]^-\simeq \ar[r]& S^{-1}K^\T(X^{\T_1}) \ar[ld]^-{\simeq} \\
 	S^{-1}K^\T(X^{\T})
 	}
 	$$
 	commutes. From the first localization theorem
 	two out of three maps on the diagram are isomorphisms. Thus third map is also isomorphism.
 \end{proof}
 \begin{atw}[relative localization formula cf. \cite{AtBo1,BeVe} for cohomology version] 
	Assume that a torus $\T$ acts on a smooth variety $X$. Let $X^{\T_1}$ denote the fixed points of subtorus $\T_1\subset \T$. Then any element $\alpha \in S^{-1}K^\T(X)$ can be recovered from its restriction to the fixed points set using formula
	$$ \alpha=\sum_{F \subset X^{\T_1}}i^F_*\frac{i^*_F\alpha}{eu^\T(\nu_F)}\,. $$
\end{atw}
\begin{proof} 
	Note that Euler class $eu^\T(\nu_F)$ is invertible in the localized $K$-theory $S^{-1}K^\T(F)$ (lemma \ref{lem:eu}).
	We use the same reasoning as in the classical localization formula. Namely, both sides of the formula are equal after restriction to the fixed points set $X^{\T_1}$. The claim follows from relative first localization theorem.
\end{proof}

\begin{atw}[relative Lefschetz-Riemann-Roch \ref{tw:LRR}, cf. theorem 5.11.7 \cite{ChGi}]
		\label{twLRR}
		Assume that a torus $\T$ acts on smooth varieties $X$ and $Y$. Consider a subtorus $\T_1\subset \T$. Let $F \subset Y^{\T_1}$ be a component of the fixed points of \hbox{the torus $\T_1$}. For any proper $\T$-equivariant map $f:X \to Y$ and element $\alpha \in S^{-1}K^\T(X)$ pushforward $f_*\alpha$ can be computed using an equality
		$$
		\frac{i_F^*f_*\alpha}{eu^{\T}(\nu_F)}=\sum_{G\subset X^{\T_1}\cap f^{-1}(F)}f^{|G}_*\frac{i_G^*\alpha}{eu^{\T}(\nu_G)}\,.
		$$
		Where the sum is indexed by the $\T_1$-fixed points components of $X$ which lie in the preimage of $F$.
	\end{atw}
 \begin{proof}
 	We use the same reasoning as in LRR theorem (remark \ref{rem:LRR}). Namely, for any \hbox{$\alpha \in S^{-1}K^\T(X)$}
 	$$
 	i_F^*f_*\alpha=i_F^*f_*\sum_{G \subset X^{\T_1}}i^G_*\frac{i^*_G\alpha}{eu^\T(\nu_G)}
 	=\sum_{G \subset X^{\T_1}}i_F^*i^{f(G)}_*f^{|G}_*\frac{i^*_G\alpha}{eu^\T(\nu_G)}
 	=eu(\nu_F)\sum_{G\subset X^{\T_1}\cap f^{-1}(F)} \frac{i^*_G\alpha}{eu^\T(\nu_G)}\,.
 	$$
 	After dividing both sides by $eu(\nu_F)$ we obtain the proof of the theorem.
 \end{proof}

\section{Appendix 2: Combinatorics} \label{s:combi}
Consider a finite set $\Omega$ and its subsets $X_1,X_2,...,X_m$. Denote by $P([m])$ the power set of the set $[m]$. For $A \in P([m])$ let
 $$X_A=\begin{cases}
 \bigcap_{i\in A}X_i \text{ when } A \neq \varnothing\,, \\
 \Omega \text{ when } A = \varnothing\,. \\
 \end{cases} $$
 The Inclusion-Exclusion formula states that
$$|\bigcup_{i=1}^m X_i|=\sum_i |X_i| -\sum_{i\neq j}|X_i\cap X_j|+\sum_{i\neq j,j\neq k,k\neq i}|X_i\cap X_j\cap X_k|-...=\sum_{A\in P([m]), A\neq \varnothing}(-1)^{|A|-1}|X_A|\,. $$
It implies that
$$|\Omega- \bigcup_{i=1}^n X_i|= \sum_{A\in P([m])}(-1)^{|A|}|X_A|\,.$$
This formula has its motivic counterpart .
Consider an algebraic $\T$-variety $\Omega$ and closed \hbox{$\T$-subvarieties} $X_1,X_2,...,X_m$. For an element $A \in P([m])$ consider subvarieties $X_A$ defined as above. Then we have an equality in the Grothendieck group $K^\T(Var/\Omega)$
$$[\Omega- \bigcup_{i=1}^m X_i\subset \Omega]= \sum_{A\in P([m])}(-1)^{|A|}[X_A \subset \Omega]\,.$$
For a smooth $\T$ variety $B$ consider $\Omega=B^k$. Denote by $\left[\binom{k}{2} \right]$ the set of unordered pairs $(i,j)$ such that $i,j \in [k], i \neq j.$ Consider  closed subvarieties
$$X_{ij}=\{(x_1,...,x_k)\in B^k| x_i = x_j\} \subset B^k  $$
for every pair $(i,j) \in \left[\binom{k}{2} \right]$. The Inclusion-exclusion formula implies that
$$[Conf_k(B) \subset B^k]=[B^k- \bigcup X_{ij}\subset B^k]=
\sum_{A\in P\left(\left[\binom{k}{2}\right]\right)}
(-1)^{|A|}[X_A \subset B^k]\,.$$
A subset $A\subset \left[\binom{k}{2}\right]$ can be visualised as a graph $G_A$ whose set of vertices is $[k]$ and whose set of edges is $A\subset \left[\binom{k}{2}\right]$.
After such interpretation the subvariety $X_A$ consist of tuples $(x_1,...,x_k) \in B^k$ such that $x_i=x_j$ when vertices $\bf{i},\bf{j}$ of the graph $G_A$ are connected by an edge. The equality relation is transitive so the set $X_A$ depends only on the partition induced on the set $[k]$ by connected components of the graph $G_A$. For a partition $P$ of the set $[k]$ consider subvariety
$$B_P=\{(x_1,...,x_k)\in B^k| x_i = x_j \text{ when } P(i)=P(j) \}\,.$$
Moreover denote by $G(P)$ the set of graphs with the set of vertices equal to the set $[k]$ whose connected components induce partition $P$ on the set $[k]$. Define
$$a(P)=\sum_{G\in G(P)} (-1)^{|E(G)|}\,.$$
It follows that
$$[Conf_k(B) \subset B^k]=
\sum_{A\in P\left(\left[\binom{k}{2}\right]\right)}
(-1)^{|A|}[X_A \subset B^k]=\sum_{P\in X_k}a(P)[B_P\to B^k]\,.$$
The numbers $a(P)$ can be computed using the following simple formula.
\begin{alemat}
	\label{lem:aP}
	$$a(P)=\prod_{P_i \in P} (-1)^{|P_i|-1}(|P_i|-1)!$$
\end{alemat}
 It is well known fact for specialists but for completeness we give a proof.
\begin{proof}
	When a given partitions $P\in X_k$ and $Q \in X_l$ and bijection $[k]\sqcup [l] \to [k+l]$ we may consider disjoint sum partition $P\sqcup Q \in X_{k+l}$.
	Both sides of the desired formula are multiplicative with respect to the disjoint sum of partitions. Thus it is enough to prove the lemma for the partition with only one element. Namely we need to show that 
	$$(-1)^{k-1}(k-1)!=\sum_{G \text{ connected graph}} (-1)^{|E(G)|}\,.$$
	Denote the left hand side by $b_k$. We proceed by induction. \\
	For $k=1$  $$b_1=1=(-1)^{k-1}(k-1)!\,.$$
	Assume $k \ge 2$. There is a bijection between the set of connected graphs which don't contain the edge $[{\bf 1},{\bf 2}]$ and the set of connected graphs which contain the edge $[{\bf 1},{\bf 2}]$ and removal of this edge won't split graph. This bijection (adding/removing edge) changes parity of $|E(G)|$. So we can sum over 
	the set of connected graphs which contain the edge $[{\bf 1},{\bf 2}]$ and removal of this edge splits graph. Such graph can be divided into two disjoint connected subgraphs (one containing the vertex {\bf 1}, and the other containing the vertex~{\bf 2})  connected by only one edge $[{\bf 1},{\bf 2}]$. Such reasoning leads us to the recursive formula for $b_k$
	$$b_k=(-1)\sum_{i=1}^{k-1} \binom{k-1}{i-1}b_ib_{k-i}\,,$$
	where $i$ corresponds to the number of vertices which are in the connected component of the vertex~${\bf 1}$ after removing the edge $[{\bf 1},{\bf 2}]$. The lemma follows from this formula by induction.
\end{proof}

	\begin{ex}
	Consider a configuration space $Conf_3(B) \subset B^3$. Denote coordinates in $B^3$ by $x,y,z$. Then
	\begin{align*}
		&[Conf_3(B)]=[B^3]-[x=y]-[x=z]-[z=x]+3[x=y=z]-[x=y=z]= \\
		&=[B_{\{\{1\},\{2\},\{3\}\}}]-[B_{\{\{1,2\},\{3\}\}}]-[B_{\{\{1,3\},\{2\}\}}]-[B_{\{\{1\},\{2,3\}\}}]+2[B_{\{\{1,2,3\}\}}] \,.
	\end{align*} 
	The term $3[x=y=z]$ corresponds to imposing two equality conditions on three elements and the term  $(-1)[x=y=z]$ to imposing all three equality conditions.
\end{ex}

\begin{rem}
	As already pointed out before, the equality for $a(P)$
	just comes from the M\"obius function of the partition lattice (see \cite{Stan}, example 3.10.4).
	There is natural action of the permutation group $S_k$ on the product $B^k$.
	We may pushforward equation~(\ref{wyr:r1}) to  the symmetric product $M^k/S_k$.  Then the class $[B_P\to M^k/S_k]$ depends only on the number partition corresponding to $P$.
	But our	formula doesn't take the natural \hbox{$S_k$-action}
	into account. For other approaches taking care of $S_k$-action see \cite{Get,Get2,MS2,MS3}
	.
\end{rem}

\newcommand{\etalchar}[1]{$^{#1}$}

\end{document}